\documentclass[11pt]{amsart}

\usepackage{mathpple}
\usepackage{amsmath,amsfonts}
\usepackage{amsrefs}
\usepackage{amsthm}
\usepackage[all]{xy}

\usepackage{hyperref}

\numberwithin{equation}{section}

\newcommand{\C}{\mathbb{C}}

\newcommand{\E}{{\mathcal E}}
\newcommand{\F}{\mathbf{F}}

\newcommand{\abracket}[1]{\left\langle#1\right\rangle}
\newcommand{\bbracket}[1]{\left[#1\right]}
\newcommand{\fbracket}[1]{\left\{#1\right\}}
\newcommand{\bracket}[1]{\left(#1\right)}

\newcommand{\mc}{\mathcal}

\newcommand{\pa}{\partial}

\renewcommand{\dbar}{\bar\pa}
\newcommand{\OO}{{\mathcal O}}

\newcommand{\BV}{Batalin-Vilkovisky }

\newcommand{\into}{\hookrightarrow}
\newcommand{\Ol}{\mathcal O_{loc}}

\newcommand{\iso}{\cong}

\newcommand{\M}{{\mathcal M}}
\renewcommand{\H}{{\mathcal H}}

\newcommand{\A}{\mathcal A}
\renewcommand{\S}{\mathcal S}
\renewcommand{\L}{\mathcal L}
\newcommand{\WL}{\widehat{\mathcal L}}

\DeclareMathOperator{\End}{End}

\DeclareMathOperator{\Sym}{Sym}
\DeclareMathOperator{\Hom}{Hom}

\DeclareMathOperator{\Tr}{Tr}

\DeclareMathOperator{\PV}{PV}

\DeclareMathOperator{\HH}{H}

\DeclareMathOperator{\Res}{Res}

\theoremstyle{plain}

\newtheorem{thm-defn}{Theorem/Definition}[section]
\newtheorem{lem}{Lemma}[section]
\newtheorem{lem-defn}{Lemma/Definition}[section]
\newtheorem{prop}{Proposition}[section]
\newtheorem{cor}{Corollary}[section]

\theoremstyle{definition}
\newtheorem{defn}{Definition}[section]

\theoremstyle{remark}
\newtheorem{rmk}{Remark}[section]

\allowdisplaybreaks[4]  %%%%%%%%%%% choose 1-4, where 4 is the strongest desire to breack

\begin{document}

 \title{Variation of Hodge structures, Frobenius manifolds and Gauge theory}
  \author{Si Li}
  \date{}

  % 封面
  \maketitle

%%%%%%%%%%%%%%%%%%%%%%%%%%%%%%
%% 前言部分
%%%%%%%%%%%%%%%%%%%%%%%%%%%%%%

\begin{abstract} We explain the homological relation between the Frobenius structure on the  deformation space of Calabi-Yau manifold and the gauge theory of  Kodaira-Spencer gravity. We show that the genus zero generating function of descendant invariants on Calabi-Yau manifolds from Barannikov's semi-infinite variation of Hodge structures is equivalent to the Kodaira-Spencer gauge theory at tree level.
\end{abstract}

  % 目录
 % \tableofcontents
  % 表格目录
%  \listoftables
  % 插图目录
%  \listoffigures

\section{Introduction}
Frobenius structure first appeared around 1983 from K.Saito's theory of higher residues and  primitive forms in singularity theory \cite{Saito-residue,Saito-primitive,Saito-universal}. The concept of Frobenius manifolds was introduced by B.Dubrovin \cite{Dubrovin} to axiomatize the two-dimensional topological field theories at genus zero, and the general structure of higher residues was reformulated by Barannikov as the notion of \emph{semi-infinite variation of Hodge structures} \cite{Barannikov-thesis,Barannikov-quantum,Barannikov-projective} in the study of mirror symmetry on Calabi-Yau manifolds. This is reinterpreted as the symplectic geometry of Lagrangian cone via Givental's loop space formalism \cite{Givental-Frobenius}.

One of the key object in Frobenius structure is the existence of local potential function. In mirror symmetry, the potential function in the A-model is given by the generating functional of genus zero Gromov-Witten invaraints, while in the B-model it is related to the special geometry on the deformation space of Calabi-Yau manifolds. Mirror symmetry at genus zero can be stated as the equivalence of potential functions in terms of flat coordinates on mirror Calabi-Yau manifolds,  which has been proven to be true for a large class of examples \cite{Givental-mirror, LLY}.

The relevant potential function is also called \emph{prepotential} in physics terminology.
It is originally observed in physics \cite{BCOV} that the prepotential on the moduli space of Calabi-Yau three-folds can be described in terms of tree diagrams from a gauge theory, which they called the \emph{Kodaira-Spencer theory of gravity}.  Motivated by this observation, it is believed that the Kodaira-Spencer theory of gravity provides an equivalent gauge theory description of B-twisted topological string on Calabi-Yau three-folds, and the higher genus analogy of potential functions are given by Feynman integrals with higher loop diagrams. Such interpretation allows us to analyze the structure of higher genus partition functions \cite{BCOV}, which has been further developed by Yamaguchi-Yau \cite{Yamaguchi-Yau} and Huang-Klemm-Quackenbush \cite{Klemm-Huang} to predict higher genus Gromov-Witten invariants on quintic three-folds.

However, the higher loop Feynman integrals usually exhibit singularities, which require renormalization to be well-defined. The mathematical analysis of such renormalization is initiated in \cite{Si-Kevin}, where the original formulation of Kodaira-Spencer gauge theory on Calabi-Yau three-folds is generalized to arbitrary dimensions with gravitational descendants included, which we call \emph{BCOV theory}.

In this paper we make the link between gauge theoretical aspects of BCOV theory and the semi-infinite variation of Hodge structures. We show that the full descendant generating function of our generalized BCOV theory \cite{Si-Kevin} at tree level is equivalent to that constructed by Barannikov \cite{Barannikov-quantum} via deformation theory. This generalizes the results in \cite{BCOV, Barannikov-Kontsevich} in full generality with gravitational descendants.

\noindent \textbf{Acknowledgement}: The author would like to thank Kevin Costello for discussions on quantum field theory,  thank Cumrun Vafa for discussions on BCOV theory, thank Kyoji Saito for discussions on primitive forms, and thank Shing-Tung Yau for discussions on Calabi-Yau geometry.

\section{Symplectic geometry of Frobenius structures}  In this section we will briefly discuss the symplectic geometry of Frobenius structures following Barannikov and Givental, and illustrate the relation with K. Saito's theory of primitive forms. A different approach along this line to primitive forms  via polyvector fields in the Landau-Ginzburg model is developed in  \cite{Li-LG, Chang-Li-Saito}.

\subsection{Semi-infinite variation of Hodge structures}
We follow the presentation in \cite{Gross-mirror} for the discussion of semi-infinite variation of Hodge structures.
\begin{defn} A \emph{semi-infinite variation of Hodge structure} (${\infty\over 2}$-VHS) parametrized by a space $\M$ is a graded locally free $\OO_\M[[t]]$-module $\E$ of finite rank with a flat (Gauss-Manin) connection
$$
       \nabla^{GM}: \E\to \Omega_M^1\otimes t^{-1}\E
$$
and a pairing
$$
     \bracket{-,-}_{\E}: \E\times \E\to \OO_M[[t]]
$$
satisfying
\begin{enumerate}
\item $\bracket{\mu_1,\mu_2}_{\E}(t)=(-1)^{|\mu_1||\mu_2|}\bracket{\mu_2,\mu_1}_{\E}(-t)$. Here we view $\bracket{\mu_1,\mu_2}_\E$ as a function of $t$, and $|\mu_i|$ is the degree of $\mu_i$.
\item $\bracket{f(t)\mu_1,\mu_2}_\E=\bracket{\mu_1, f(-t)\mu_2}_\E=f(t)\bracket{\mu_1,\mu_2}_\E$, $\forall f\in \C((t))$
\item $V\bracket{\mu_1,\mu_2}_\E=\bracket{\nabla^{GM}_{V}\mu_1,\mu_2}_\E+\bracket{\mu_1, \nabla^{GM}_{V}\mu_2}_\E$, $\forall$ vector field $V$ on $\mc M$.
\item The induced pairing
$$
        \E /t\E\otimes_{\OO_{\mc M}}\E/t\E \to \OO_{\mc M}
$$
is non-degenerate.
\end{enumerate}
The ${\infty\over 2}$-VHS is called \emph{miniversal} if there is a section $s$ of $\E$ such that
$$
       t \nabla^{GM}s: T_\M\to \E/t\E, \quad X\to t\nabla_X^{GM} s
$$
is an isomorphism.
\end{defn}

\begin{rmk}
If we decompose into components
$$
    \bracket{-,-}_{\E}=\sum_{k\geq 0}t^k  \bracket{-,-}_{\E}^{(k)}
$$
then $\bracket{-,-}_{\E}^{(k)}$ plays the role of \emph{higher residue pairing} \cite{Saito-residue}.
\end{rmk}

\begin{rmk}
We will not discuss the extra data like the Euler vector field etc for the purpose of the current paper. They play important role and we refer to the literature for further explanation.
\end{rmk}

\subsection{Symplectic geometry}
We consider the symplectic geometry of ${\infty\over 2}$-VHS. Let
$$
   \mc H= \fbracket{s\in \Gamma\bracket{\M, \E\otimes _{\C[[t]} \C((t))}| \nabla^{GM} s=0}
$$
be the space of flat sections. $\mc H$ is a free $\C((t))$-module with a symplectic pairing defined by
$$
   \omega(\mu_1, \mu_2)\equiv \Res_{t=0} \bracket{\mu_1,\mu_2}_{\E}dt, \quad \forall \mu_1, \mu_2\in \mc H
$$

Given $p\in \M$, and let $\E_p$ be the fiber of $\E$ at $p$, we have a natural embedding
$$
  \E_p\into \mc H
$$
by solving the flat equation with initial condition at $p$. Then $\E_p$ is Lagrangian with respect to the symplectic form $\omega$. We can view $\E_p$ as a moving family of Lagrangian linear subspaces inside $\mc H$.

\begin{defn}
A polarization of $\mc H$ is a Lagrangian subspace $\mc H_-\subset \mc H$ preserved by $t^{-1}$ such that
$$
  \mc H=\mc H_-\oplus \E_p, \quad \forall p\in \M
$$
\end{defn}

We will focus on the situation when $\M$ is a formal scheme with base point $0$. In this case we only need to require that $\mc H=\mc H_-\oplus \E_0$ at the base point.

\begin{lem} $\forall \alpha, \beta\in \mc H_-$, we have
$
   \bracket{\alpha, \beta}_\E \in t^{-2}\C[t^{-1}]
$
\end{lem}
\begin{proof}
This follows from the fact that $\mc H_-$ is Lagrangian and preserved by $t^{-1}$.
\end{proof}

\begin{cor}\label{pairing}
Let $\alpha, \beta\in t\mc H_-\cap \E_p$, then $\bracket{\alpha,\beta}_\E\in \C$.
\end{cor}
\begin{proof} $\alpha, \beta\in \E_p$ implies
    $$
       \bracket{\alpha, \beta}_{\E}\in \C[[t]]
    $$
    and $\alpha, \beta\in t\mc H_-$ implies that
    $$
          \bracket{\alpha, \beta}_{\E}\in \C[t^{-1}]
    $$
The corollary follows.
\end{proof}

\subsection{Frobenius structure}
We recall Barannikov's construction of Frobenius manifolds from semi-infinite variation of Hodge structures.
We refer to \cite{Manin-Frobenius} for basics on Frobenius structures.

\subsubsection{Semi-infinite period map}
Let $\fbracket{\M, \E, \nabla^{GM}, \bracket{-,-}_\E}$ be a miniversal ${\infty\over 2}$-VHS with base point $0\in \M$, and $\mc H_-$ be a chosen polarization such that
$$
  \mc H=\mc H_-\oplus \E_0
$$
We assume that there exists an element $
   \Omega_0\in t \mc H_-
$
such that the section
$$
  s=\bracket{\Omega_0+ \H_-}\cap \E \in \Gamma\bracket{\M, \E}
$$
yields the miniversality. This is called a \emph{semi-infinite period map} \cite{Barannikov-quantum}.

The following lemma follows easily from the definition
\begin{lem} For any vector field $X$ on $\M$,
$$
t\nabla^{GM}_Xs\in t\mc H_-\cap \E
$$
where
$$
   t\mc H_-\cap \E=\fbracket{\alpha\in \Gamma(\M, \mc E)| \alpha(p)\in t\mc H_-, \forall p\in \M}
$$
It defines an isomorphism of $\OO_\M$-modules
$$
    t\nabla^{GM}s: T_\M\to  t\mc H_-\cap \E
$$
\end{lem}

\begin{lem}Given any vector field $X\in T_\M$,
$$
   \nabla^{GM}_X:  t\mc H_-\cap \E\to  \bracket{t\mc H_-\cap \E}\oplus \bracket{\mc H_- \cap t^{-1}\E}
$$
\end{lem}
\begin{proof}
Any element of $t\mc H_-\cap \E$ can be represented by $t\nabla^{GM}_Y s$ for some $Y\in T_\M$. It follows from the property of $\nabla^{GM}$ and the flatness of $\Omega_0$ that
$$
   \nabla^{GM}_X\bracket{t\nabla^{GM}_Y s}\in t\mc H_-\cap t^{-1} \E=\bracket{t\mc H_-\cap \E}\oplus \bracket{\mc H_- \cap t^{-1}\E}
$$

\end{proof}

The properties of the above two lemmas are the key structures for defining a \emph{primitive form} in the case of singularity theory \cite{Saito-primitive}.

This allows us to construct a Frobenius structure on $\mc M$ as follows. We will identify $T_\M$ with  $t\mc H_-\cap \E$ via $t\nabla^{GM}s$, then above decomposition gives rise to
$$
   \nabla^{GM}=\nabla+{1\over t}A
$$
where $\nabla$ is a connection on $T_\M$, and
$$
   A: T_\M\to \End(T_\M)
$$
By Corollary \ref{pairing}, the pairing from ${\infty\over 2}$-VHS induces a metric $g$ on $T_\M$ since
$$
    g=\bracket{-,-}_\E:  \bracket{t\mc H_-\cap \E}\otimes \bracket{t\mc H_-\cap \E}\to \OO_\M
$$

\begin{prop}The triple $\fbracket{\nabla, A, g}$ defines a Frobenius structure on $\M$.
\end{prop}
\begin{proof} This follows from the flatness of $\nabla^{GM}$ and its compatibility with the pairing $\bracket{-,-}_\E$.
\end{proof}

\subsubsection{Flat coordinate}
The induced Frobenius structure can be concretely described in terms of flat coordinates. The semi-infinite period map induces a morphism
$$
  \Psi: \M \to t\mc H_-/ \mc H_-, \quad p \to ts(p)
$$
which is a local isomorphism.

There is a natural identification $t\H_-/\H_-\simeq t\H_-\cap \E_0$. Let's choose a basis $\fbracket{\Delta_a}$ of $t\H_-\cap \E_0$, and $\fbracket{\tau^a}$ be the dual linear coordinates.  It gives rise to a local coordinate on $\M$ via the local isomorphism
$$
   \psi: \M\to t\H_-\cap \E_0\simeq t\mc H_-/ \mc H_-
$$
which is called the \emph{flat coordinates}. In terms of flat coordinates, the semi-infinite period map is given by
$$
  s=\Omega_0+ {1\over t}\sum_{a}\tau^a \Delta_a+O(\tau^2)
$$
where the higher order term $O(\tau^2)$ lies in $t^{-1}\mc H_-$.  It follows that
$$
  t\pa_a s\in \bracket{\Delta_a+\mc H_-}\cap \E
$$
and $\fbracket{t\pa_a s}$ forms  a basis of $t\mc H_-\cap \E$. Here $\pa_a$ is short for the derivation with respect to $\tau^a$.

\begin{lem}
In terms of flat coordinates, $g_{ab}:=g(t\pa_as, t\pa_b s)\in \C$ is a constant.
\end{lem}
\begin{proof} Since $t\pa_a s, t\pa_b s\in t\mc H_-$ and $\Delta_a\in t\mc H_-$
$$
   g(t\pa_a s, t\pa_b s)=\lim_{t\to \infty}(t\pa_a s, t\pa_b s)_\E=\lim_{t\to \infty}(\Delta_a, \Delta_b)_\E
$$
which is now a constant since $\Delta_a, \Delta_b$ are flat.
\end{proof}
This lemma explains the name of "flat coordinates". On the other hand, since the leading term in $s$ is linear in $\tau$,
$$
  t^2\pa_a\pa_b s\in t\mc H_-\cap \E
$$
Since $\fbracket{t\pa_a s}$ forms  a basis of $t\mc H_-\cap \E$, there exists $A_{ab}^c(\tau)$ such that
$$
 t^2\pa_a\pa_b s= \sum_c A_{ab}^c t\pa_c s
$$
which describes the product in the Frobenius structure
$$
  \pa_a \circ \pa_b =\sum_c A_{ab}^c \pa_c
$$

\section{Calabi-Yau geometry}
In this section we discuss Barannikov's construction of semi-infinite variation of Hodge structures in Calabi-Yau geometry and Givental's interpretation of $J$-function in the B-model.

\subsection{Polyvector fields}

Let $X$ be a compact Calabi--Yau manifold of dimension $d$. Follow \cite{Barannikov-Kontsevich} we consider the space of polyvector f\/ields on $X$
\[
\PV(X)=\bigoplus_{0\leq i,j \leq d}\PV^{i,j}(X),\qquad
\PV^{i,j}(X)= \A^{0,j} \big(X, \wedge^i T_X\big).
\]
Here $T_X$ is the
holomorphic tangent bundle of $X$, and $\A^{0,j} (X, \wedge^i T_X) $
is the space of smooth $(0,j)$-forms valued in $\wedge^i T_X$.   $\PV(X)$ is a dif\/ferential bi-graded
commutative algebra; the dif\/ferential is the operator
\[
\bar\partial : \ \PV^{i,j} (X) \rightarrow \PV^{i,j+1} (X),
\]
and the algebra structure arises from wedging polyvector f\/ields. The
degree of elements of $\PV^{i,j}(X)$ is $i + j$. The graded-commutativity says that
\[
   \alpha\beta=(-1)^{|\alpha||\beta|}\beta\alpha,
\]
where $|\alpha|$, $|\beta|$ denote the degree of $\alpha$, $\beta$ respectively.

The Calabi--Yau condition implies that there exists a nowhere vanishing holomorphic volume form
\[
\Omega_X \in H^0(X, K_X),
\]
which is unique up to a multiplication by a constant. Let us f\/ix a choice of $\Omega_X$. It induces an isomorphism between the space of polyvector f\/ields and dif\/ferential forms
\begin{gather*}
  \PV^{i,j}(X)   \stackrel{\vdash \Omega_X}{\iso} \A^{d-i,j}(X), \qquad
    \alpha  \to\alpha \vdash \Omega_X,
\end{gather*}
where $\vdash$ is the contraction map.

The holomorphic de Rham dif\/ferential $\pa$ on dif\/ferential forms def\/ines an operator on polyvector f\/ields via the above isomorphism, which we still denote by
\[
\partial :  \ \PV^{i,j}(X) \to \PV^{i-1,j}(X),
\]
i.e.
\[
(\partial \alpha) \vdash \Omega_X \equiv  \partial ( \alpha \vdash \Omega_X), \qquad \alpha\in \PV(X).
\]
Obviously, the def\/inition of $\pa$ doesn't depend on the choice of $\Omega_X$. It induces a bracket on polyvector f\/ields
\[
  \fbracket{\alpha, \beta}=\pa\bracket{\alpha\beta}-\bracket{\pa\alpha}\beta-(-1)^{|\alpha|}\alpha\pa\beta,
\]
which associates $\PV(X)$ the structure of \emph{Batalin--Vilkovisky algebra}.

We def\/ine the \emph{trace map} $\Tr: \PV(X)\to \C$ by
\[
   \Tr(\alpha)=\int_X \bracket{\alpha\vdash \Omega_X}\wedge \Omega_X.
\]
Let $\abracket{-,-}$ be the induced pairing
\begin{gather*}
    \PV(X)\otimes \PV(X)  \to \C,\qquad
        \alpha\otimes \beta   \to \abracket{\alpha,\beta}\equiv \Tr\bracket{\alpha\beta}.
\end{gather*}
It's easy to see that $\bar\partial$ is skew self-adjoint for this pairing and $\pa$ is self-adjoint.

\subsection{Symplectic geometry and Lagrangian cone}\label{Calabi-Yau-symplectic}
We consider Givental's loop space formalism in the B-model. Let
$$
   S(X)=\PV(X)((t))[2]
$$
with the decomposition
$$
  S_+(X)=\PV(X)[[t]][2], \quad S_-(X)=t^{-1}\PV(X)[t^{-1}][2]
$$
Here $[2]$ is the conventional shifting of degree by $2$.  Let $Q=\dbar+t\pa$, and
$$
   \mc H=H^*(S(X), Q)
$$
There is a natural isomorphism
$$
   \Gamma_{\Omega_X}: \PV(X)((t))\stackrel{\iso}{\to} \A^{*,*}(X)((t))
$$
which sends $t^k \alpha^{i,j}\in t^k\PV^{i,j}(X)$ to
$$
   \Gamma_{\Omega_X}\bracket{t^k \alpha^{i,j}}=t^{k+i-1}\alpha^{i,j}\vdash \Omega_X
$$
Under $\Gamma_{\Omega_X}$, the operator $Q$ becomes the ordinary de Rham differential
$$
   \Gamma_{\Omega_X}(Q)= d
$$
and we can identify
$$
  \Gamma_{\Omega_X}: \mc H\to H^*(X, \C)((t))
$$
This will play the role of the flat structure. We define a pairing on $S(X)$
$$
  \bracket{-,-}: S(X)\times S(X)\to \C((t))
$$
by
$$
     \bracket{f(t)\alpha, g(t)\beta}=f(t)g(-t)\Tr\bracket{\alpha\beta}
$$
It's easy to see that it descends to a pairing on $\mc H$. This allows us to define the symplectic form on $\mc H$ via
$$
  \omega(\mu_1, \mu_2):= \Res_{t=0}\bracket{\mu_1, \mu_2}dt, \quad \mu_1, \mu_2\in \mc H
$$

Let $\tilde M$ be the space of gauge equivalent solutions of
$$
  \tilde M=\fbracket{\mu \in S_+(X)| Q\mu+{1\over 2}\fbracket{\mu, \mu}=0}/\sim
$$
associated to the differential graded Lie algebra $\fbracket{\S_+(X), Q, \fbracket{,}}$. There is a natural embedding
$$
  \tilde M\to \mc H: \quad \mu \to t-te^{\mu/t}
$$
in the sense of formal geometry via its functor of points on Artinian graded rings (Strictly speaking we should look at the formal derived stack associated to differential graded Lie algebra, and the above morphism corresponds to the morphism on Maurer-Cartan functors). We let $\L_X$ be the image of $\tilde M$.
\begin{prop} $\L_X$ is a formal Lagrangian sub-manifold of $\mc H$ around $0$. The tangent space at $\mu\in \L_X$ is given by
$$
T_\mu \L_X=\{\alpha e^{\mu/t}| \alpha\in H^*(S_+(X), Q+\fbracket{\mu,-})\}
$$
\end{prop}

We define
$$
  \E_0= H^*(S_+(X), Q)
$$
Then $\E_0$ is the tangent space of $\L_X$ at $\mu=0$.

\subsection{Semi-infinite variation of Hodge structures}
 Let $\mc H_-\subset \mc H$ be a choice of polarization. It defines the decomposition
$$
  \mc H=\mc H_-\oplus \E_0
$$
such that $t^{-1}: \mc H_-\to \mc H_-$. Let $\pi_+: \mc H\to \E_0$ be the associated projection. 

\begin{prop}[\cite{Barannikov-quantum}] $\mc L_X\cap t \mc H_-$ is a smooth formal scheme around $\mu=0$ such that
$$
    T_{0}\bracket{\mc L_X\cap t \mc H_-}=\mc E_0\cap t\mc H_-
$$
\end{prop}

It follows that the natural projection on the Lagrangian $\pi_+: \mc L_X\to \E_0$ is an isomorphism.  We follow Barannikov \cite{Barannikov-quantum} to associate a ${\infty\over 2}$-VHS on $\mc E_0\cap t\mc H_-$, which can be identified with the extended deformation space of $X$ \cite{Barannikov-Kontsevich}. We consider the locally free sheaf $\mc E$ on $\mc E_0\cap t\mc H_-$, whose fiber over $\mu\in \mc E_0\cap t\mc H_-$ is
$$
   \mc E_\mu=T_{\pi_+^{-1}(\mu)} \mc L_X\subset \mc H
$$
It becomes a varying family of Lagrangian subspaces in $\mc H$.  Let $\nabla^{GM}$ be the trivial connection on $\mc H$. It's easy to see that
$$
  \nabla^{GM}: \mc E\to \Omega^1_{\mc E_0\cap t\mc H_-}\otimes t^{-1}\mc E
$$

\begin{prop}[\cite{Barannikov-quantum}] $\fbracket{\mc E_0\cap t\mc H_-, \E, \nabla^{GM}, g=\bracket{-,-}}$ defines a miniversal ${\infty\over 2}$-VHS on $\mc E_0\cap t\mc H_-$.
\end{prop}

To find the Frobenius structure, we consider the intersection \cite{Givental-Frobenius}
$$
   J=t\mc H_-\cap \mc L_X
$$
which is \emph{Givental's J-function}. Let
$$
  \pi_0: t \H_-\to \E_0\cap t\H_-
$$
be the projection to the component of $\E_0\cap t\mc H_-$, then
$$
  \pi_0: J\to \E_0\cap t\H_-
$$
is an isomorphism. In this way we view $J$ as defining a varying section of $t\H_-$ parametrized by $\E_0\cap t\H_-$. Note that
$$
   1-J/t \in \bracket{1+\mc H_-}\cap \E
$$
as a section defined on $\E_0\cap t\H_-$, and $J$ plays the role of semi-infinite period map. It gives an isomorphism
$$
    \nabla^{GM}: T_{\mc E_0\cap t\mc H_-}\to t\mc H_- \cap \E, \quad V\to \nabla^{GM}_V J
$$
which implies that there induces a Frobenius structure on $\mc E_0\cap t\mc H_-$, and the linear coordinates on $ \E_0\cap t\mc H_-$ are precisely the flat coordinates.

\subsection{Potential}
The potential function can be expressed in term of $\mc L_X$. Let's choose a basis $\fbracket{\Delta_a}$ of $t\H_-\cap \E_0$, and $\fbracket{\tau^a}$ be the dual flat linear coordinates.  We adopt the same notation in the previous section. The Frobenius structure in flat coordinates is described by the equation
$$
   t\pa_a\pa_b J=A_{ab}^c\pa_c J
$$
and we let
$$
A_{abc}=A_{ab}^d g_{dc}
$$
where $g_{ab}=\bracket{t\pa_a s, t\pa_b s}$ is the metric in flat coordinates.

The splitting $\mc H=\mc H_-\oplus \E_0$ gives a natural identification
$$
  \mc H= T^*(\E_0)
$$
Let $\F_0\in \OO(\mc E_0)$ be the generating function for the Lagrangian cone $\mc L_X$
$$
   \mc L_X=\mbox{Graph}\bracket{d\F_0}
$$
and we let
$$
f_0=\F_0|_{\mc E_0\cap t\mc H_-}
$$
be the restriction to $\mc E_0\cap t\mc H_-$. The following interpretation is due to Givental \cite{Givental-Frobenius}
\begin{prop}
$f_0$ gives the potential function for the Frobenius structure, i.e.
$$
   A_{abc}=\pa_a\pa_b\pa_c f_0
$$
\end{prop}
\begin{proof} We write
$$
  J=\tau+ B(\tau)+C(\tau)
$$
where $\tau=\sum_a \tau^a\Delta_a$, $B(\tau)\in \mc H_-\cap t^{-1}\E_0$, $C(\tau)\in t^{-1}\mc H_-$. $B(\tau), C(\tau)$ are higher order terms in $\tau$. By the definition of $f_0$ and $J$,
\begin{align*}
  \pa_a f_0&=\bracket{B(\tau), \Delta_a}=\Res_{t=0}\bracket{J-\tau, \Delta_a}dt
\end{align*}
where we use the fact from Corollary \ref{pairing} that the pairing
$$
  \bracket{-,-}: \bracket{t^k\mc H_-\cap t^{k-1}\E_0}\otimes  \bracket{t^m\mc H_-\cap t^{m-1}\E_0}\to \C t^{k+m-2}
$$
It follows that
\begin{align*}
  \pa_a\pa_b f_0=\Res_{t=0}\bracket{\pa_aJ-\Delta_a, \Delta_b}dt
\end{align*}

Therefore
\begin{align*}
  \pa_a\pa_b\pa_c f_0=\Res_{t=0}\bracket{\pa_a\pa_bJ, \Delta_c}dt=\Res_{t=0}\bracket{\pa_a\pa_bJ, \pa_cJ}dt=A_{abc}
\end{align*}
\end{proof}

\subsubsection{The polarizations}
We give some remarks on the polarizations. Under the isomorphism $\Gamma_{\Omega_X}$ in Section \ref{Calabi-Yau-symplectic},
$$
   \Gamma_{\Omega_X}\bracket{\E_0}=\sum_{p}t^{d-p+1}F^p H^*(X, \C)\subset H^*(X, \C)((t))
$$
corresponds to the Hodge filtration. A choice of polarization is equivalent to a filtration on $H^*(X, \C)$ which splits the Hodge filtration.

A natural choice is given by the complex conjugate filtration. If we choose a K\"{a}hler metric, and let
$$
  \HH\subset \PV(X)
$$
be the space of harmonic polyvector fields, then
$$
  \mc H\iso \HH((t))
$$
via the harmonic projection, and the decomposition of $\mc H$ is given by
$$
   \E_0=\HH[[t]], \quad \H_-=t^{-1}\HH[t^{-1}]
$$

The $J$-function can be obtained as follows. Let $\tilde \mu$ be the $\HH$-parametrized universal solution of the Maurer-Cartan equation
$$
    \dbar \tilde\mu+{1\over 2}\fbracket{\tilde\mu,\tilde\mu}=0, \quad \tilde\mu \in \PV(X)\otimes \OO(\HH)
$$
subject to the constraint $\pa \tilde \mu=0$.  The existence is proved in \cite{Todorov-CY, Barannikov-Kontsevich}. The solution is given by power series
$$
   \tilde\mu= \tau+ \mbox{higher order terms}
$$
where $\tau$ parametrizes $\HH$. Then we see that as formal manifold
$$
   J=t-t e^{\tilde \mu/t}\in t\H_-
$$
is the $J$-function with respect to the complex conjugate filtration, and the map
$$
  \pi_0: J\to \HH,  t-t e^{\tilde \mu/t}\to \tau
$$

Another choice relevant to mirror symmetry is the monodromy splitting filtration around the large complex limit over the moduli space. Different choices of polarizations will be related via a change of coordinates \cite{Barannikov-thesis}.

\section{Kodaira-Spencer gauge theory } We now give the gauge theoretical description of the structure of semi-infinite variation of Hodge structure on Calabi-Yau manifold. This is first discovered in physics by Bershadsky, Cecotti, Ooguri and Vafa \cite{BCOV} known as the \emph{Kodaira-Spencer gauge theory} and the mathematical aspect is developed in \cite{Si-Kevin, Li-thesis, Si-mirror}. A finite dimensional toy model has also appeared in \cite{Losev-Shadrin, Shadrin-BCOV}. See \cite{Si-survey} also for a short introduction.

\subsection{Lagrangian cone and BCOV action}
We consider the Lagrangian cone in the previous section at the chain level. Let
$$
   \widehat{\L}_X=\fbracket{t-te^{\mu/t}|\mu\in S_+(X)}\subset S(X)
$$
be a submanifold of $S(X)$ in the sense of formal geometry \cite{Si-Kevin}. $S(X)$ is a formal symplectic space with symplectic form $\omega$, and $\widehat{\L}_X$ is a Lagrangian such that the dilaton shift \cite{Givental-Frobenius} $\widehat{\L}_X-t$ is a cone. The decomposition
$$
  S(X)=S_+(X)\oplus S_-(X)
$$
gives a formal identification
$$
  S(X)=T^*(S_+(X))
$$

\begin{defn}
The classical BCOV action $S^{BCOV}$ is a formal functional on $S_+(X)$ defined to be the generating functional of $\WL_X$, i.e.
$$
   \WL_X=\mbox{Graph}(dS^{BCOV})
$$
\end{defn}

\begin{prop}[\cite{Si-Kevin}]
\[
    S^{BCOV}(\mu)=\Tr \abracket{e^{\mu}}_0,
\]
where $\abracket{-}_0: \Sym(\PV(X)[[t]])\to \PV(X)$ is the map given by intersection of $\psi$-classes over the moduli space of marked rational curves
\[
    \big\langle\alpha_1 t^{k_1},\dots, \alpha_n t^{k_n}\big\rangle_0=\alpha_1\cdots \alpha_n \int_{\overline{M}_{0,n}}\psi_1^{k_1}\cdots\psi_n^{k_n}=\binom{n-3}{k_1,\dots,k_n}\alpha_1\cdots \alpha_n.
\]
\end{prop}

\begin{rmk}
If we restrict $\mu$ to be elements from $t^0\PV(X)$,  then $S^{BCOV}$ becomes cubic and we recover the \emph{Yukawa coupling} of the usual Kodaira-Spencer gauge theory introduced in \cite{BCOV}. Strictly speaking, $S^{BCOV}$ defines the interaction part of BCOV theory, while the quadratic kinetic term is non-local and degenerate \cite{BCOV}. However, the propagator and Feynman diagrams are still well-defined, which allows a well-behaved perturbative renormalization \cite{Si-Kevin}. This has a natural generalization to Landau-Ginzburg B-model \cite{Li-LG}. 
\end{rmk}

%%%%%%%%%%%%%%%%%%%%%%%%%%%%%%%%%%%%%%%%%%%%%%%%%%%%%%%%%
\iffalse
The homogeneous degree $n$ part of the formal function $S^{BCOV}$ will be denoted by
\[
   D_n S^{BCOV}: \ S_+(X)^{\otimes n}\to \C, \qquad    D_n S^{BCOV}\bbracket{\mu_1,\dots, \mu_n}=\bracket{{\pa\over \pa \mu_1}\cdots {\pa\over \pa \mu_n}}S^{BCOV}(0).
\]
Then $S^{BCOV}(\mu)=\sum\limits_{n\geq 3}{1\over n!} D_n S^{BCOV}\bbracket{\mu^{\otimes n}}$ and
\[
  D_n S^{BCOV}\bbracket{\alpha_1 t^{k_1},\dots, \alpha_n t^{k_n}}=\binom{n-3}{k_1,\dots,k_n}\Tr(\alpha_1\cdots \alpha_n).
\]
The lowest component is cubic: it's non-zero only on $t^{0}\PV(X)\subset \E(X)$
\[
    D_3 S^{BCOV}\bbracket{\alpha_1,\alpha_2,\alpha_3}=\Tr(\alpha_1\alpha_2\alpha_3), \qquad \alpha_i\in \PV(X),
\]
which is literally called the \emph{Yukawa-coupling}.
\fi
%%%%%%%%%%%%%%%%%%%%%%%%%%%%%%%%%%%%%%%%%%%%%%%%%%%%%%%%%%

\subsection{Potential function and BCOV action}

\subsubsection{Functionals and dual spaces}  Let $S_+(X)^{\otimes n}$ be the completed projective tensor product of n copies of $S_+(X)$. It can be viewed as the space of smooth polyvector f\/ields on $X^n$ with a formal variable $t$ for each factor.  Let
\[
   \OO^{(n)}(S_+(X))=\Hom\bracket{S_+(X)^{\otimes n}, \C}_{S_n}
\]
denote the space of continuous linear maps (distributions), and the subscript $S_n$ denotes ta\-king~$S_n$ coinvariants. $\OO^{(n)}(S_+(X))$ will be the space of homogeneous degree n functionals on the space of f\/ields~$S_+(X)$, playing the role of~$\Sym^n(V^{\vee})$ in the case of f\/inite-dimensional vector space~$V$. We will also let
\[
   \Ol^{(n)}(S_+(X)) \subset \OO^{(n)}(S_+(X))
\]
be the subspace of local functionals, i.e.\ those of the form given by the integration of a Lagrangian density
\[
   \int_X \mathcal L(\mu), \qquad \mu\in S_+(X).
\]

\begin{defn}
The algebra of functionals $\OO(S_+(X))$ on $S_+(X)$ is def\/ined to be the product
\[
    \OO(S_+(X))=\prod_{n\geq 0}  \OO^{(n)}(S_+(X)),
\]
and the space of local functionals is def\/ined to be the subspace
\[
   \Ol(S_+(X))=\prod_{n\geq 0} \Ol^{(n)}(S_+(X)).
\]
\end{defn}
In particular, $S^{BCOV}\in \Ol(S_+(X))$. There are similar definitions for $\OO(S(X))$ and $\Ol(S(X))$.

We will also let $\overline{S(X)}$ be the distributional sections of $S(X)$. Using the symplectic form $\omega$, we can identify
\[
   \overline{S(X)}\iso \OO^{(1)}(S(X))
\]
with a natural embedding
\[
   \Psi: S(X)\into \OO^{(1)}(S(X))
\]

The differential $Q$ induces a derivation on $\OO(S(X))$. The ellipticity of the $\dbar$-operator implies that $\Psi$ is a quasi-isomorphism, and it extends to quasi-isomorphic embedding
\[
   \Psi: \prod_{n}\Sym^n\bracket{S(X)}\into \OO(S(X))
\]

We will also use $\OO(\H)$ to denote the formal functions on $\H=H^*(S(X), Q)$. The symplectic pairing $\omega$ allows us to identify
\[
   \OO(\H)\iso \prod_n\Sym^n(\H)
\]
and $\Psi$ gives a natural isomorphism
\[
   \bar\Psi: \OO(\H)\stackrel{\simeq}{\to} H^*(\OO(S(X)), Q)
\]

\subsubsection{Classical master equation and $L_\infty$-structure}
Let
\[
  \pi_\pm: S(X)\to S_{\pm}(X)
\]
be the projections. Then
\[
  \pi_+: \WL_X\to S_+(X)
\]
is an isomorphism of formal graded manifolds \cite{Si-Kevin}.
\begin{lem}[\cite{Si-Kevin}]\label{Q-tangent} The derivation $Q$, viewed as a vector field on $S(X)$, is tangent to $\WL_X$.
\end{lem}

Let $\hat Q=\pi_{+*}(Q)|_{\WL_X}$ be the push-forward of the vector field $Q$ on $\WL_X$. It becomes an odd nilpotent vector field on $S_+(X)$, which is equivalent to a $L_\infty$-structure.

\begin{defn} We define the kernel $K$ by the distributional section of $\PV(X)\otimes \PV(X)$
\[
   K=(\pa\otimes 1)\delta
\]
where $\delta$ is the delta-function distribution supported on the diagonal, characterized by the property
\[
  \bracket{\delta, \alpha\otimes \beta}=\Tr\bracket{\alpha\beta}, \forall \alpha, \beta\in \PV(X)
\]
\end{defn}

It's easy to see that $K$ is symmetric, and defines a Poisson bracket as usual
\[
   \{,\}: \Ol(S_+(X))\otimes \OO(S_+(X))\to \OO(S_+(X))
\]
Note that since $K$ is distribution valued, one of the input of the Poisson bracket is required to be local. We refer to \cite{Si-Kevin} for more careful explanation about the construction.

\begin{prop}\cite{Si-Kevin} The induced $L_\infty$-structure on $S_+(X)$
\[
   \hat Q: \OO(S_+(X))\to \OO(S_+(X))
\]
is given by
$$
   \hat Q=Q+ \fbracket{S^{BCOV},-}
$$
where the first term $Q$ is that induced dually from the derivation $Q: S_+(X)\to S_+(X)$.
\end{prop}

The nilpotent nature $\hat Q^2=0$ is equivalent to the \emph{classical master equation}
$$
   QS^{BCOV}+{1\over 2}\fbracket{S^{BCOV}, S^{BCOV}}=0
$$
In physics terminology, $Q+\fbracket{S^{BCOV},-}$ defines the gauge symmetry in the \BV formalism. In this way, $S^{BCOV}$ defines a meaningful interacting gauge theory for polyvector fields on Calabi-Yau manifolds.

\subsubsection{Lagrangian cone and generating functional} Let $I_{\WL_X}\subset \OO(S(X))$ be the formal ideal of the Lagrangian $\WL_X$. If we identify $\WL_X$ with $S_+(X)$ under the projection $\pi_+$, then we have an induced exact sequence
$$
   0\to I_{\WL_X}\to \OO(S(X))\stackrel{\eta}{\to} \OO(S_+(X))\to 0
$$
$I_{\WL_X}$ is preserved by $Q$ by Lemma \ref{Q-tangent}. The map $\eta$ is generated topologically on the generators:
$$
    \eta(\mu)=\begin{cases} \Psi(\mu) & \mbox{if}\ \mu \in t^{-1}\PV(X)[t^{-1}] \\
         \pa_\mu S^{BCOV} & \mbox{if}\ \mu \in \PV(X)[[t]]\end{cases}
$$
where ${\pa_{\mu}S^{BCOV}}$ is the derivation of $S^{BCOV}$ with respect to $\mu$ in the sense of formal geometry.
 If we associate the differential $Q$ to $I_{\WL_X}, \OO(S(X))$, and associate $\hat Q=Q+\fbracket{S^{BCOV},-}$ to $\OO(S_+(X))$, then the above sequence is an exact sequence of complexes.

\begin{prop} We have an exact sequence
$$
  0\to H^*(I_{\WL_X}, Q)\to H^*(\OO(S(X)),Q)\stackrel{\eta}{\to} H^*(\OO(S_+(X)), \hat Q)\to 0
$$
\end{prop}
\begin{proof}
Let's fix a choice of K\"{a}hler metric on $X$, $\dbar^*$ be the adjoint of $\dbar$ and $\Delta=\dbar\dbar^*+\dbar^*\dbar$ be the Laplacian on $\PV(X)$. We denote by
$$
   \HH=\{\mu \in \PV(X)| \Delta \mu=0 \}\subset \PV(X)
$$
be the space of harmonic polyvector fields. A simple application of spectral sequence shows that we have the natural isomorphism
$$
    \Sym(\HH((t)))\stackrel{\Phi}{\to}  H^*(\OO(S(X)),Q), \quad \Sym(t^{-1}\HH[t^{-1}]) \stackrel{\Psi}{\to} H^*(\OO(S_+(X)), \hat Q)
$$
Here we use the fact that if $\mu\in t^{-1}\HH[t^{-1}]$, then $\fbracket{S^{BCOV}, \Psi(\mu)}=0$. The following diagram commutes
$$
  \xymatrix{\Sym(t^{-1}\HH[t^{-1}])\ar[r]^{\Phi}\ar[rd]^{\Psi} & H^*(\OO(S(X)),Q)\ar[d]^{\eta}\\
          & H^*(\OO(S_+(X)), \hat Q)
  }
$$
which implies the surjectivity of $\eta$ on cohomology.
\end{proof}
If we identify $H^*(\OO(S(X)), Q)=\OO(\mc H)$, then $H^*(I_{\WL_X}, Q)$ is the ideal defining the Lagrangian $\L_X\subset \H$ at the cohomology level.

We would like to find the corresponding generating functional with respect to the complex conjugate filtration.

To find the kernel of $\eta$, we consider $\mu\in \HH[[t]]$, viewed as an element of $ H^*(\OO(S(X)),Q)$. Then
$$
  \eta(\mu)={\pa_{\mu}S^{BCOV}}
$$
To find the representative of $\eta(\mu)$ via harmonic elements, we need to find the homotopic inverse of
$$
  \Psi: \Sym(t^{-1}\HH[t^{-1}]) {\to} H^*(\OO(S_+(X)), \hat Q)
$$
This is essentially an application of homological perturbation lemma. Consider the operator
$$
  G= {\dbar^*\over \Delta}
$$
Since $1-\bbracket{Q, G}$ is the harmonic projection, it gives a homotopic contraction for the embedding
$$
   \Sym\bracket{t^{-1}\HH[t^{-1}]}\into \bracket{\OO(S_+(X)), Q}
$$
If we view $\hat Q$ as the homological perturbation of $Q$ by $\fbracket{S^{BCOV},-}$, then the homotopic inverse of $\Psi$ is given by
$$
   \Psi^{-1}= \Pi \sum_{k\geq 0}\bracket{\{S^{BCOV},-\}G}^k
$$
where in the last step $\Pi$ is the Harmonic projection. It follows that the kernel of $\eta$ is generated by
$$
   \mu-\Pi \bracket{\sum_{k\geq 0}\bracket{\{S^{BCOV},-\}G}^k \pa_\mu S^{BCOV}}, \quad \mu\in  \HH[[t]]
$$

\begin{defn}
The genus zero partition function $\F^{BCOV}_0\in \OO(\HH[[t]])$ using BCOV theory is defined to be
$$
   \F^{BCOV}_0=\sum_{\Gamma: \mbox{Tree}} {W_\Gamma(P, S^{BCOV})\over |Aut(\Gamma)|}
$$
where $P$ is the kernel of the operator ${\dbar^*\pa\over \Delta}$, $W_\Gamma(P, S^{BCOV})$ is the Feynman diagram integral with $S^{BCOV}$ as the vertices and $P$ as the propagator, $Aut(\Gamma)$ is the size of the automorphism group as graphs. The summation is over all connected tree diagrams with external edges where we put harmonic polyvector fields.
\end{defn}

\begin{rmk}
$P$ is the full propagator of Kodaira-Spencer gauge theory \cite{BCOV, Si-Kevin}.
\end{rmk}

The usual trick on Feynman diagrams gives
$$
\pa_\mu \F^{BCOV}_0=\Pi\bracket{\sum_{k\geq 0}\bracket{\{S^{BCOV},-\}G}^k \pa_\mu S^{BCOV}}
$$
and the above calculation implies that
$$
  \mu-\pa_\mu \F^{BCOV}_0 \in \ker(\eta), \quad \quad \mu\in  \HH[[t]]
$$
This proves the following

\begin{prop}
The generating functional of the Lagrangian $\L_X\subset \H$ with respect to the complex conjugate filtration is given by $\F^{BCOV}_0$.
\end{prop}

Therefore the tree level partition function of BCOV theory is equivalent to the data of semi-infinite variation of Hodge structures on Calabi-Yau manifolds. In particular, the potential function for the corresponding Frobenius structure is given by the restriction of $\F_0^{BCOV}|_{\HH}$ to $\HH$.

\subsection{Quantization}
Finally we briefly remark the quantization approach in \cite{Si-Kevin} following the line of the above consideration. We take the point of view of the standard Weyl quantization.

Let $\mc W(\H)$ be the Weyl algebra of the symplectic space $\H$, which is the pro-free algebra generated by $\mc H^{\vee}$ over $\C[[\hbar]]$, subject to the relation that
$$
   [a,b]=\hbar \omega^{-1}(a,b), \forall a,b \in \H^{\vee}
$$
where $\omega^{-1}\in \wedge^2{\H}$ is the inverse of $\omega$. $\mc W(\H)$ is viewed as the non-commutative deformation of $\OO(\H)$ parametrized by $\hbar$, defining the sheaf of functions on a non-commutative space $\H^{\hbar}$ such that $\H^{\hbar}|_{\hbar=0}=\H$. The classical geometry of BCOV theory is described by the Lagrangian $\L_X\subset \mc H$, and the quantization is given by a $\C[[\hbar]]$-flat non-commutative deformation $\L_X^{\hbar}\subset \H^{\hbar}$ such that $\L_X^{\hbar}\cap \H=\L_X$.

Similar to the genus zero case, the deformation $\L_X^{\hbar}$ can be described by the classical BCOV action $S^{BCOV}$ and Feynman integrals with higher loop graphs. However, since the propagator $P={\dbar^*\pa\over \Delta}$ is singular along the diagonal, the higher loop Feynman diagrams are usually divergent. This is the usual ultra-violet difficulty in quantum field theory, and the standard solution is via renormalization. The general framework for the renormalization of BCOV theory along Costello's homological techniques \cite{Costello-book} is described in \cite{Si-Kevin}. In \cite{Si-Kevin}, the renormalization/quantization is interpreted in term of Fock space, and the ideal sheaf of $\L_X^{\hbar}$ in $\mc W(\H)$ defines a vector in the Fock space.

%%%%%%%%%%%%%%%%%%%%%%%%%%%%%%
%% 正文部分
%%%%%%%%%%%%%%%%%%%%%%%%%%%%%%
%\mainmatter

%%%%%%%%%  Appendix  %%%%%%%%%%%%%%%
%  \appendix
%\section{Apendix A}

\bibliography{biblio}

\address{\tiny DEPARTMENT OF MATHEMATICS AND STATISTICS, BOSTON UNIVERSITY, 111 CUMMINGTON MALL, BOSTON} \\
\indent \footnotesize{\email{sili@math.bu.edu}}

\end{document}